\documentclass[12pt]{amsart}

\usepackage{graphicx,epstopdf}
\usepackage{color}
\usepackage{hyperref}
\hypersetup{colorlinks=false,pagebackref=true}

\usepackage[a4paper]{geometry}
\usepackage{amsmath}
\usepackage{amssymb}
\usepackage{amsthm}

\usepackage{framed}
\newcommand{\R}{\mathbb{R}}
\newcommand{\N}{\mathcal{N}}
\newcommand{\Arg}{\mathop{\mathrm{Arg}}}
\newcommand{\rank}{\mathrm{rank}\,}
\newtheorem{definition}{Definition}
\newtheorem{lemma}{Lemma}
\newtheorem{theorem}{Theorem}
\newtheorem{example}{Example}
\newtheorem{proposition}{Proposition}
\newtheorem{corollary}{Corollary}
\renewcommand{\int}{\mathrm{int}\,}
\newcommand{\diam}{\mathrm{diam}\,}
\def\g {\mathcal G}

\title[Optimization methods on smooth and proximally smooth manifolds]{Error 
bound conditions and convergence of optimization methods on smooth 
and proximally smooth manifolds}

\author{M. V. Balashov, A. A. Tremba}

\address{V. A. Trapeznikov Institute of Control Sciences of Russian Academy of Sciences,
65 Profsoyuznaya street, Moscow 117997, Russia.
balashov73@mail.ru, tremba@ipu.ru}

\begin{document}

\renewcommand{\thefootnote}{ }
\footnote{The work was supported by Russian Science Foundation
(Project 16-11-10015)}
\renewcommand{\thefootnote}{\arabic{footnote}}


\maketitle

\begin{abstract}
We analyse the convergence of the gradient projection algorithm,
which is finalized with the Newton method, to a stationary point
    for the problem of nonconvex constrained optimization $\min_{x \in
S} f(x)$ with a proximally smooth set  $S = \{x \in \R^n : g(x) =
0 \}, \; g : \R^n \rightarrow \R^m$ and a smooth function $f$. We
propose new Error bound (EB) conditions for the gradient
projection method which lead to the convergence domain of the
Newton method. We prove that these EB conditions are typical for
a wide class of optimization problems. It is possible to reach
high convergence rate of the algorithm by switching to the Newton
method.
\end{abstract}


\section{Introduction}

Problems of constrained optimization on manifolds are complex
because it is impossible to demand  convexity-like conditions
from a function defined on a manifold. One should use more
flexible, in comparison with convexity, conditions for the
function and for the set. Using these conditions we plan to
analyse the convergence of the gradient projection algorithm (GPA)
and the combined algorithm, including the GPA and the Newton
method (NM). The principle is well known and can be found for
example in \cite{zabotin-chernyaev2001}, see also the
bibliography in \cite{zabotin-chernyaev2001}. Nevertheless there
are no estimates of the rate of convergence. The rate of
convergence was estimated in a particular case for some
proximally smooth sets in \cite{bpt}.

We consider the following finite-dimensional optimization problem
\begin{equation} \label{eq:main}
\min_{\displaystyle x \in S} f(x),
\end{equation}
with a proximally smooth set $S$.
We shall consider $S$ in the form of  the system of  $m$ equations $g_i(x) =
0, \, i = 1,..., m$, or by the vector function $g : \R^n \rightarrow \R^m,\, m < n$. In other words
\begin{equation} \label{eq:S-definition}
S = \{x \in \R^n : g(x) = 0\}.
\end{equation}
Further we also shall assume that the set $S$ is compact and the function $f$ is smooth.

The real Stiefel manifold $S_{n,k}$ is very important example of
the set $S$ (\ref{eq:S-definition}): $S_{n,k}=\{ X\in\R^{n\times
k}\ :\ X^TX=I_k\}$, $k\le n$, $I_k$ is the  $k\times k$ identity
matrix.

Our aim is to find a point of minimum for the function $f$ on the set $S$ or, at least,
a stationary point. We propose to use the GPA as a base method and  to switch to Newton's method
 in a small neighborhood
of a stationary point. The latter will accelerate the convergence rate of the algorithm.

We want to recall some general difficulties for nonconvex problems:

\begin{itemize}
\item[(i)]
the metric projection is not a singleton and not continuous (as a set-valued function),
\item[(ii)]
there could be stationary points, which are not extremums,
\item[(iii)]
the gradient of a differentiable function is not a monotone operator.
\end{itemize}

We consider proximally smooth sets $S$ {\cite{Vial, Clarke}}
because of item (i). The metric projection on such set is a
singleton for any point which is sufficiently close to the set.

The Error bound (EB) condition will be an important technical tool
for the problem under consideration. EB conditions are widely
spreaded in unconstrained optimization
\cite{karimi-etal2016,drusvyatskiy-lewis2016} and recently they
actively penetrate into problems of constrained optimization
\cite{bpt,gao-etal2016,liu-etal2015}. We propose to formulate the
EB condition for the set of stationary points but not for the set
of minimizers, see section
~\ref{sec:nondegeneracy-and-wEB-definition}. This condition
replaces convexity assumptions for the function and for the set
and gives the convergence of the method (and of the GPA in
certain cases). Thus we solve question (ii).

We also consider functions with Lipschitz continuous gradient.
For any function $f$ with Lipschitz continuous gradient $f'$ with
constant $L_{1}$ the function $f(x)+\frac12 L_{1}\| x\|^{2}$ is
convex. This property helps us to solve difficulty (iii).

The proposed method can be used for minimization of a twice
continuously differentiable functions on a smooth and proximally
smooth compact manifolds without edge. We need not the Riemannian
metric, geodesics and retraction with the help of the exponential
mapping \cite{absil-etal2008}. We shall use only the standard metric
projection onto the set  $S$.

The paper has the following structure.

Base results about the GPA and conditions of extremum are
gathered in section ~\ref{sec:basics}: choice of the step-size
and definition of a stationary point. Algorithm of minimization
with switching from the GPA to the modified NM is described in
section~\ref{sec:main-results}. New EB conditions are also
defined in the same section. We give examples of problems with
new EB conditions: minimization of a quadratic function on a
sphere or on the Stiefel manifold. We introduce the notion of
nondegenerate problem  for (\ref{eq:main}). We prove that new EB
conditions are typical, they take place for any nondegenerate
problem. In contrast with the standard Newton method \cite[Ch. 2,
\S 1]{Aubin}, \cite[Ch. 1, \S 1.4]{Bertsekas} and some other
algorithms \cite[Ch. 4]{Bertsekas}, \cite[Ch. 8, \S 2]{Polyak}
which converge locally, the proposed algorithm converges for any
initial point $x_{0}\in S$ and its iterations belong generally to
the set $S$.

For the convenience of readers we have collected  proofs in
Appendix at the end of the article.

\section{Base notations and methods}
\label{sec:basics}

Let $\R^{n}$ be an $n$-dimensional Euclidean space with the inner
product $(x,y)$ for all $x,y\in\R^{n}$ and with the norm $\|
x\|=\sqrt{(x,x)}$ for all $x\in\R^{n}$.

Further we demand twice continuous differentiability of the
functions $f(\cdot), g_i(\cdot)$ ($f, g_i \in \mathcal{C}^2$),
and Lipschitz continuity of the second derivatives  $f''(\cdot),
g''_i(\cdot)$.
We treat the gradient  $f'$ etc. as a column.

Suppose that the function $f$ is Lipschitz with constant  $L_0$,
and its gradient  $f'$ is also Lipschitz with constant $L_1$.

Denote by
\begin{equation} \label{eq:jakobi-matrix}
g'(x) = (g_1'(x)\vdots g_2'(x) \vdots\dots\vdots g_m'(x) )^{T}
\in \R^{m \times n}
\end{equation}
the \it Jacobi matrix \rm for the function $g(x)$. We demand the
standard \it full rank condition \rm $\rank g'(x) = m$ on the
set  $S$.

Assume that the manifold  $S$ is \it compact and without edge. \rm
\footnote{This condition can be weakened and we can demand
compactness for the intersection of some \it lower level set  \rm for $f$ with $S$. Let  $x_0\in S$ be
an initial starting point in context of numerical methods. Then
one can assume that the intersection of the edge of the set $S$ and the lower
level set $\mathcal{L}_f(f(x_0))=\{x\in\R^{n}\, :\, f(x)\le
f(x_{0})\}$ is empty and the intersection of the set $S$ and the lower
level set $\mathcal{L}_f(f(x_0))$ is compact. }.

Denote by  $T_x$ the \it tangent subspace \rm at the point $x \in
S$. It is characterized with the help of the Jacobi matrix
\eqref{eq:jakobi-matrix} by the formula  $T_x = \{v \in \R^n :
g'(x) v =0\}$.

\it The metric projection \rm of a point $x \in \R^n$ onto a set
$Q \subseteq \R^n$ is defined as follows
$$
P_Q(x) = \{y \in Q : \|x - y\| = \rho(x, Q)\},
$$
where $\rho(x, Q) = \inf_{y \in Q} \|x - y\|$ is the \it distance
function\rm.

We shall use vertical stacking of column vectors: $[a, b] = (a^T,
b^T)^T$. The same notations will be applied for matrices $a$, $b$
of particular sizes. For a number $t>0$ we denote by $\lceil
t\rceil$ the minimal natural number with $t\le \lceil t\rceil$.

The \it metric projecting operator \rm onto the tangent subspace
$T_x$ is given by the matrix  \cite[Ch. 7, \S 2, Formula
(7)]{Polyak}
\[
P_{T_x} = I_n - g'^T(x)(g'(x) g'^T(x))^{-1}g'(x) = I_n - g'^T(g' g'^T)^{-1}g',
\]
the metric projection of a point  $y \in \R^n$ onto a subspace
$T_x$ is denoted by $P_{T_x} y$. The metric projecting operator $P_{T_{x}}$ is defined for all
$x\in S$ by the full rank condition for matrix
\eqref{eq:jakobi-matrix}.
We shall omit the dependence of an expression on
argument if this dependence is obvious from a context.

For a set  $S \subset \R^n$ and a number $R > 0$ define the set
$$
U_S(R) = \{x \in \R^n : 0 < \rho(x, S) < R \}
$$
that is a layer (or "tube") around the set $S$.

An important requirement for the set $S$ in our work is its
proximal smoothness (also known as prox-regularity or weak
convexity), that is characterized by constant of proximal
smoothness $R>0$.

\begin{definition}[{\cite{Vial, Clarke}}]
\label{def:proximally-smooth} A closed set $S\subset\R^{n}$ is called
\emph{proximally smooth} with constant $R$ if the distance
function  $\rho(x, S)$ is continuously differentiable on $U_S(R)$.
\end{definition}

Existence, uniqueness of $P_{S}x$ for all $x\in U_S(R)$ and continuity%
\footnote{In a finite dimensional space continuity of the mapping
$U_S(R)\ni x\to P_{S}x$ can be omitted. This follows from
uniqueness and upper semicontinuity of
the metric projection \cite[Ch. 3, \S 1, Proposition 23]{Aubin}.}
of the mapping $U_S(R)\ni x\to P_{S}x$ in a real Hilbert space are
equivalent conditions for proximal smoothness of the set $S$ with
constant $R$. In other words the set $S$ has the Chebyshev layer of size $R$.

For a point  $x$ of a proximally smooth set  $S$ the  \emph{cone
of proximal normals} (or simply --- \emph{normal cone}) is
defined as
\[
\N(S, x) = \big\{ p \in \R^{n} : \exists t > 0, P_S(x + t p) =
\{x\} \big\}.
\]
 This cone coincides with any other cone to the proximally smooth set
$S$ at the point $x\in S$
 (in particular with cones of Clarke and Bouligand) \cite{poliquin-etal2000,Thibault}.
For our situation, when $S$ is given by the system
\eqref{eq:S-definition}, the normal cone coincides with the
orthogonal subspace to the tangent subspace $T_{x}$, i.e.  $\N(S,
x) = \{ g'(x)^T w : w \in \R^{m}\}$.

It is obvious that the Euclidean sphere of radius $R$ is
proximally smooth with constant $R$.

\begin{example}[{\cite{Vial}}]
Suppose that $g : \R^n \rightarrow \R^1$,
 $g$ is a Lipschitz function with constant $L_g$ and
there exists $\ell > 0$ such that for any $x \in S$ we have
$\|g'(x)\| \geq \ell$. Then the set $S = \{x \in \R^n : g(x) =
0\}$ is proximally smooth with constant $R = \ell / L_g$.
\end{example}

Sometimes one can calculate constant of proximal smoothness using the supporting principle
for proximally smooth sets, see
\cite{bookVol2}.

\begin{proposition}
\label{prop:stiefel-is-proximally-smooth} The Stiefel manifold
$S_{n,k}=\{ X\in\R^{n\times k}\ :\ X^TX=I_k\}$ of any dimensions
is proximally smooth with constant $R=1$. This constant is the
largest possible. \rm See the proof in
section~\ref{sec:projection-on-simple-sets} of the Appendix.
\end{proposition}

\subsection{Stationary conditions}

For the problem \eqref{eq:main} with a proximally smooth set  $S$
points of minimum are characterized by the next necessary
condition: the anti-gradient $-f'(x)$ at such point $x\in S$
belongs to the normal cone $\N(S,x)$ \cite[Appendix 5.1]{bpt}. We
shall call such point \emph{stationary} and denote their set by
\[
\Omega = \{x \in S : -f'(x) \in \N(S, x)\}.
\]

If the set is given by the system (\ref{eq:S-definition}), the
stationary condition is equivalent to the equality $P_{T_x} f'(x)
= 0$ (or $\|P_{T_x} f'(x) \| = 0$).

Consider the Lagrange function with the Lagrange multiplier
$\lambda \in \R^m$:
$$
f(x) + \sum_{i=1}^m \lambda_i g_i(x).
$$
The stationary condition can be written with the help of
derivative of the Lagrange function with respect to the extended
variable $z = [x, \lambda] \in \R^{n + m}$, i.e. in the form $F(z)
= 0$, where
\begin{equation} \label{eq:F-definition}
F(z) = F(x, \lambda) = \begin{bmatrix}
f'(x) + g'(x)^T \lambda \\
g(x)
\end{bmatrix}.
\end{equation}
The Hessian matrix of the Lagrange function coincides with
$F'(z)$ and has the form
\begin{equation} \label{eq:F'-definition}
F'(z) = \begin{bmatrix}
f''(x) + (g'(x)^T)'_x \lambda & g'(x)^T \\
g'(x) & 0
\end{bmatrix}
=
\begin{bmatrix}
f''(x) + \sum_{i=1}^m \lambda_i g''_i(x) & g'(x)^T \\
g'(x) & 0
\end{bmatrix}
.
\end{equation}

Note some relationship between the derivative of the Lagrange function
at the point $[x,\lambda]$ and the metric projection
$P_{T_{x}}f'(x)$ for a point $x\in S$. For any $x \in S$ define
$\lambda_x$ by the formula
\begin{equation} \label{eq:lambda_x-definition}
\lambda_x = \arg \min_\lambda \|F(x, \lambda)\|
= \arg \min_\lambda \|f'(x) + g'(x)^T \lambda\|
= -(g'g'^T)^{-1}g'f'.
\end{equation}
We get the equality $f' + g'^T \lambda_x = (I - g'^T(g'
g'^T)^{-1})f' = P_{T_x}f'$. Note also that $\|F(x, \lambda_x)\| =
\|P_{T_x}f'\|$. The variable $\lambda_{x}$ depends on $x$.
We shall
use the notation
\begin{equation} \nonumber %\label{eq:F_x-definition}
F_x(x) \doteq F(x, \lambda_x)=P_{T_x}f'(x)
\end{equation}
to distinguish the functions $F(z)$ and $F(x, \lambda_x)$.

Further we shall use the notation $F'(x, \lambda_x)$ that means
\begin{equation} \label{eq:F'_x-definition}
F'(x, \lambda_x) \doteq F'([x, \lambda_x]) \equiv F'(z)\big|_{z = [x, \lambda_x]}.
\end{equation}
Notice that the last expression is not a derivative of the
function $F_x(x)$ on $x$. By the Lipschitz condition for $g''$ and
$f''$ the function $F'(x, \lambda_x)$ is also Lipschitz
in a compact neighborhood of $S$ with some
constant $L_{1,Fx}$ .

\subsection{The gradient projection algorithm}

Consider some results about convergence of the GPA
\begin{equation} \label{eq:proj-grad}
x_{k+1} = P_S(x - \gamma f'(x_k)).
\end{equation}
for a proximally smooth set  $S$ and a function  $f(\cdot)$ with
 Lipschitz continuous gradient \cite{bpt}.

The idea of applying the GPA for nonconvex problem is the next one.
Let $x_k \in S$. Then we can choose the step-size
$\gamma$ with the property  $x_k - \gamma f'(x_k)\in U_S(R)\cup S$.
Projection $P_S(x_k - \gamma f'(x_k))$ exists and it is unique by the definition of proximally smooth set.
Next, we can adopt the step-size in such way that the sequence
$f(x_k)$ will be monotonically decreasing.

\begin{theorem}[{\cite[Theorem 1]{bpt}}]
\label{thm:projected-gradient-convergence} Suppose that  $S$
is a proximally smooth set with constant $R$ and $x_0 \in S$
is a starting point. Assume that a function  $f : \R^n \to \R$ is
Lipschitz with constant $L_0$, and its gradient is also Lipschitz
with constant  $L_1$. Then for any fixed step-size $0 < \gamma <
\min\{\frac{1}{L_1}, \frac{R}{L_0}\}$ the GPA
\eqref{eq:proj-grad} converges to the set of stationary points
$\Omega$, i.e. $\lim\limits_{k\to\infty}\rho (x_{k},\Omega)=0$.
Moreover,
$$
f(x_{k+1})
\leq f(x_k) - \frac{1}{2}\Big(\frac{1}{\gamma} - L_1\Big) \|x_{k+1} - x_k\|^2.
$$
\end{theorem}

Note that we can estimate the number of steps which is necessary
for finding a stationary point with any a priori precision.

\begin{corollary}
\label{cor:projected-gradient-number-of-steps}
Suppose that under conditions of Theorem~\ref{thm:projected-gradient-convergence}
we know the value $\Delta f = f(x_0) - \inf_{x \in S} f(x)$,
and the set $S$ is given by the system \eqref{eq:S-definition}.
Then for any $\varepsilon > 0$ we can find a natural number
$i: 0 \leq i \leq
\left\lceil \frac{2 \Delta f (1 + \gamma L_1)^2}%
{\varepsilon^2 \gamma (1 - \gamma L_1)}%
\right\rceil
$ with
$\|P_{T_{x_i}} f'(x_i)\| \leq \varepsilon$.
\end{corollary}
The proof of the Theorem and Corollary can be found in section  \ref{sec:projected-gradient-convergence-proof}
of the Appendix.

Finding the metric projection of a point onto the set is an
important part of the gradient projection algorithm. If the set
has simple structure then the metric projection can be easily
found, e.f. for Euclidean sphere or the Stiefel manifold, see
Proposition \ref{prop:stiefel-is-proximally-smooth} in
section~\ref{sec:projection-on-simple-sets} of the Appendix.

In paper \cite[algorithm GPA2]{bpt}   we consider an algorithm
for finding some easily computing quasi-projection instead of the
 metric projection for the case of one equation $g : \R^n
\rightarrow \R$.
\subsection{The Newton method}

Now we formulate sufficient conditions for convergence of the Newton
method for the equation $F(z) = 0$.
We shall assume for simplicity that the function
$F(z)$ is continuously differentiable everywhere.
\begin{proposition}[{\cite[Theorem X.4.1]{kolmogorov-fomin2004}, see also \cite[Ch. 2, \S 1]{Aubin}, \cite{Bertsekas}}]
\label{prop:mod-newton-convergence} Suppose that the derivative
$F'(z)$ is Lipschitz continuous with constant $L_{1,F}$, the
matrix $F'(z_0)$ is invertible at the point $z_0 = [\widehat{x}_0,
\lambda_0]$ and the condition
\begin{equation} \label{eq:newton-convergence-condition}
L_{1,F} \, \|F'(\widehat{x}_0, \lambda_0)^{-1}F(\widehat{x}_0,
\lambda_0)\| \cdot \|F'(\widehat{x}_0, \lambda_0)^{-1}\| <
\frac{1}{4}
\end{equation}
holds. Then the modified Newton method  \eqref{eq:modified-newton}
converges to a solution $z=z^*=[x_*,\lambda_*]$ of the equation
$F(z) = 0$. Moreover $z^*$ is a unique solution in the ball with
centerpoint $z_{0}=[\widehat{x}_0,\lambda_{0}]$ and radius
$r=Kt_{0}$. Here we have $K=\| F'(z_0)^{-1}F(z_{0})\|$, and
$t_{0}\in (0,2]$ is the smaller root of the equation
$ht^{2}-t+1=0$, where $h=L_{1,F}K\| F'(z_0)^{-1}\|$.

Besides, a linear rate of convergence takes place for iterations of the modified Newton method:
$$
\|\widehat{x}_k - x_*\| \leq  2^{1-k} \|F'(\widehat{x}_0,
\lambda_0)^{-1}F(\widehat{x}_0, \lambda_0)\| \leq  2^{1-k}
\|F'(\widehat{x}_0, \lambda_0)^{-1}\| \cdot \|F(\widehat{x}_0,
\lambda_0)\|
$$
and $[\widehat{x}_k,\lambda_k]\in B_r(z_0)$. Note that it's
sufficient to require Lipschitz continuity of $F'(z)$ on the ball
$B_{2K}(z_0)$.
\end{proposition}

\section{Main results}
\label{sec:main-results}

\subsection{Combined algorithm: the GPA and the NM}

We shall assume that we have information about Lipschitz constant  $L_0$ of the function
$f$ and Lipschitz constant  $L_1$ of the gradient $f'$.
Suppose also that we know constant of proximal smoothness  $R$
for the set $S$.

The next Algorithm depends on some real positive constant $C > 0$,
we shall specify its value below.

\begin{framed}
\textbf{Combined algorithm: GPA + NM}

\begin{enumerate}
\item[]
\textbf{Starting conditions and parameters:} Given constant $C>0$.\\
 Choose arbitrarily $x_0 \in S = \{x \in \R^n
: g(x) = 0\}$. Put $k = 0$ and take $0 < \gamma <
\min\{\frac{1}{L_1},
\frac{R}{L_0}\}$.

\item[Step 1]
\textbf{Gradient projection algorithm, GPA:} \\
Do \eqref{eq:proj-grad}, increasing $k$:
\[
x_{k+1} = P_S(x_k - \gamma f'(x_k)).
\]
In the case

\begin{equation} \label{eq:switching-condition}
\|P_{T_{x_k}}f'(x_k)\| < C,
\end{equation}
(or in the case
\eqref{eq:switching-condition-on-delta-x})
go to  Step~2.

\item[Step 2]
\textbf{Preparation for the NM:} \\
Define the initial point $\widehat{x}_{0} = x_{k}$,
$\lambda_0
= \lambda_{\widehat{x}_0}
= -(g'(\widehat{x}_0)g'(\widehat{x}_0)^T)^{-1}g'(\widehat{x}_0)f'(\widehat{x}_0)$,
put $k = 0$.

\item[Step 3]
\textbf{Modified Newton method, NM:} \\
Do steps of the \emph{modified} NM for the equation $F(z) =
F(\widehat{x}, \lambda) = 0$, increasing $k$:
\begin{equation} \label{eq:modified-newton}
\begin{array}{l}
\begin{bmatrix}
\widehat{x}_{k+1} \\
\lambda_{k+1}
\end{bmatrix}
=
\begin{bmatrix}
\widehat{x}_{k} \\
\lambda_{k}
\end{bmatrix}
- F'\Big(\begin{bmatrix}
\widehat{x}_{0} \\
\lambda_{0}
\end{bmatrix}
\Big)^{-1}
F\Big(\begin{bmatrix}
\widehat{x}_{k} \\
\lambda_{k}
\end{bmatrix}
\Big) = \\
\!\!\!\!\!\!\!\!\!\!\!\!\!\!\!
=
\begin{bmatrix}
\widehat{x}_{k} \\
\lambda_{k}
\end{bmatrix}
-
\begin{bmatrix}
f''(\widehat{x}_0) + \sum_{i=1}^m (\lambda_k)_i g_i''(\widehat{x}_0) &
g'(\widehat{x}_0)^T\\
g'(\widehat{x}_0) & 0
\end{bmatrix}^{-1}
\begin{bmatrix}
f'(\widehat{x}_k) + g'(\widehat{x}_k)^T \lambda_k \\
g(\widehat{x}_k) % \frac{1}{2}(\|x_k\|^2 - 1)
\end{bmatrix}.
\end{array}
\end{equation}

Algorithm should be stopped with the help of some stop criteria,
e.g. after a given number of steps and so on, see further.
\end{enumerate}

\end{framed}
The value $C$ is a nontrivial parameter of the Algorithm. Its
calculation requires information about some additional constants,
see section~\ref{sec:main-results}.

We want to pay attention on some peculiarities of the Algorithm.

\begin{enumerate}
\item
Firstly, at the initial step (GPA) the sequence
 $x_k$ belongs to the manifold $S$.
Condition $\gamma < R/L_0$ guarantees the inclusion $x_k - \gamma
f'(x_k) \in S \cup U_S(R)$ and uniqueness of the metric
projection. Another condition $\gamma < 1/L_1$ guarantees that
the sequence $f(x_k)$ is monotonically decreasing, see the proof
of Theorem~\ref{thm:projected-gradient-convergence} in
section~\ref{sec:projected-gradient-convergence-proof} of the
Appendix. The maximum number of steps in this phase is also
explicitly estimated there.

\item
We can use a simpler condition for switching the GPA phase to the
NM phase instead of condition~\eqref{eq:switching-condition},
namely
\begin{equation} \label{eq:switching-condition-on-delta-x}
\|x_k - x_{k-1}\| \leq \frac{\gamma}{1 + \gamma L_1} C, \; k \geq
1.
\end{equation}
This condition needs computing of a simple value $\|
x_{k}-x_{k-1}\|$ instead of $P_{T_{x_k}}f'(x_k)$. The
admissibility of this condition follows from the estimate $\|x_k
- x_{k-1}\|  (\frac{1}{\gamma} + L_1)\ge
\|P_{T_{x_k}}(f'(x_k))\|$, inequality
\eqref{eq:relation-betweed-rho-and-delta-x} and the limit
$\|x_{k} - x_{k-1}\| \rightarrow_{k \to \infty} 0$ (see the proof
of Corollary~\ref{cor:projected-gradient-number-of-steps} in
section~\ref{sec:projected-gradient-convergence-proof} of the
Appendix).

\item
We use the
\emph{modified NM} in the second phase.
It needs   \emph{a unique} computing of the inverse matrix
 $F'(\widehat{x}_0, \lambda_0)^{-1}$. The points
$\widehat{x}_k \in \R^n$ do not necessarily belong to the set  $S$
and $\lambda_k$ is an independent variable.
\end{enumerate}

The main criterion for stopping the Algorithm (the NM phase) is the
inequality
$$
\|P_{T_{\widehat{x}_k}} f'(\widehat{x}_k)\| \leq \varepsilon,
$$
or estimate for the distance to the set of stationary points
$$
\rho(\widehat{x}_k, \Omega) \leq \varepsilon,
%\;\;  = \|x_k - x_*\| x_* \in \Omega
$$
for some $\varepsilon > 0$. The last is achieved after given
number of steps for the NM, this number can be determined if
appropriate constants are known.

The fulfillment of switching  condition
\eqref{eq:switching-condition} or
\eqref{eq:switching-condition-on-delta-x} for any $C > 0$ is
guaranteed by Theorem~\ref{thm:projected-gradient-convergence}
about convergence of the GPA on a proximally smooth set and by
Corollary~\ref{cor:projected-gradient-number-of-steps}.

Next consider the conditions that provide the convergence of the
NM at the second phase of the Algorithm.

\subsection{Nondegenerate problems and error bound conditions}
\label{sec:nondegeneracy-and-wEB-definition}

\begin{definition} \label{def:nondegenerate-definition}
We shall call the problem \eqref{eq:main} with $f, g \in
\mathcal{C}^2$ \emph{\bf nondegenerate}, if for all stationary
points $x_*\in\Omega$ the matrix $F'(x_*,\lambda_{x_*})$ is
invertible. Norms of all inverse matrices are bounded from above
by a value  $\sigma_0>0$:
\begin{equation} \label{eq:sigma-0-definition}
\|F'(x_*, \lambda_{x_*})^{-1}\| \leq \sigma_0, \;\; \forall x_* \in \Omega.
\end{equation}
\end{definition}
This definition allows to use the NM in neighborhoods of
stationary points, see
Lemma~\ref{lem:sigma-around-stationary} below%
\footnote{ We can treat
$\sigma_0$ as \emph{minimal}
singular value  of matrices  $F'(z)\big|_{z = [x_*,
\lambda_{x_*}]}, \; x_* \in \Omega$
(it coincides with minimal by absolute value eigenvalue,
$\Lambda(\cdot)$ denotes the spectrum of a matrix):
$\sigma_0
\geq \max_{x_* \in \Omega} \|F'(x_*, \lambda_{x_*})^{-1}\|
%= \frac{1}{\displaystyle \min_{x_* \in \Omega} \sigma_{\min} (F'(x_*, \lambda_{x_*}))}
= \big(\min_{x_* \in \Omega} \sigma_{\min} (F'(x_*, \lambda_{x_*}))\big)^{-1}
%= \frac{1}{\displaystyle \min_{x_* \in \Omega, \lambda \in \Lambda (F(x_*, \lambda_{x_*}))} |\lambda|}.
= \big(\min_{x_* \in \Omega, \lambda \in \Lambda (F(x_*,
\lambda_{x_*}))} |\lambda|\big)^{-1}. $ }.

\begin{lemma} \label{lem:finite-number-of-stationary-points-in-nondegenerate}
If the set $S$ is compact and the problem \eqref{eq:main} is nondegenerate then
the number of stationary points is finite, $\Omega =
\{x_{j}\}_{j=1}^{J}$.
\end{lemma}
The proof can be found in section~\ref{sec:finite-number-of-stationary-points-in-nondegenerate}
of the Appendix.

The next important definition characterizes relationship between
stationary points and $P_{T_x}f'(x)$ at any point $x\in S$.

\begin{definition} \label{def:wEB-definition}
We shall say that problem  \eqref{eq:main} satisfies the
\emph{\bf tangent Error Bound condition} (or {\bf tEB}), if there
exists a positive value
 $\mu
> 0$ with
\begin{equation} \label{eq:weak-error-bound}
\mu \, \rho(x, \Omega)\le \|P_{T_x}f'(x)\| =  \|F_x(x)\|, \;\;
\forall x \in S.
\end{equation}
\end{definition}

For a point $x\in S$ and $\gamma>0$ denote by  $\g_{\gamma}(x) =
\frac{x-P_{S}(x-\gamma f'(x))}{\gamma}$ the \it gradient mapping \rm
at the point $x$ for problem (\ref{eq:main})
\cite{Nesterov}. It is clear that the gradient mapping is related with one step of the GPA.

\begin{definition}
\label{def:wEB-2-definition} We shall say that problem
\eqref{eq:main} satisfies the  \emph{\bf gradient Error Bound
condition} (or {\bf gEB}) with constant $\nu>0$, if there exists
$\gamma_0>0$ such that for all  $0 < \gamma < \gamma_0$ and for
all  $x \in S$ we have
$$
\nu \rho (x,\Omega)\le \|\g_{\gamma}(x)\|.
$$
\end{definition}

Definition \ref{def:wEB-2-definition} was formulated (without gEB notation) in the paper \cite{MSbReview}.

By the proof of Proposition~\ref{cor:gradient-mapping-greater-F'}
(section~\ref{sec:projected-gradient-convergence-proof} of the
Appendix) in the case of smooth and proximally smooth manifold
the tEB condition entails the gEB condition with constant $\nu =
\frac{\mu}{1+L_{1}\gamma_{0}+\mu\gamma_{0}}$ and $\gamma_0 =
\min\{\frac{1}{L_1}, \frac{R}{L_0}\}$. Conditions tEB and gEB are
equivalent in the case of smooth and proximally smooth manifold
$S$. The proof is bulky and we omit it.

In contrast with  Error Bound conditions in unconstrained optimization
$$
\|f'(x)\| \geq \mu \, \rho(x, X_{\min}), \;\;
X_{\min} = \Arg \min_{x \in \R^n} f(x),
$$
or equivalent Lezanski-Polyak-Lojasiewicz condition\footnote{%
Sometimes is called the Polyak-Lojasiewicz  condition  or the Kurdyka-Lojasiewicz
condition $\mu \|f'(x)\|^{\alpha} \geq f(x) - f_{\min}$,
$\alpha\ge 1$. } \cite{karimi-etal2016}, we use the distance from
a point $x\in S$ to the set of
 \emph{stationary} points $\Omega$
and the value $\|P_{T_x} f'(x)\|$
instead of $\| f'(x)\|$.

Consider few examples.

\begin{example}
\rm One can explicitly calculate constant $\mu$ in the tEB
condition for a quadratic form on the unit Euclidean sphere.
\begin{lemma}
\label{lem:wEB-on-sphere}
The  tEB condition fulfills with constant
$$
\mu = \min_{i \neq j} |\lambda_i - \lambda_j|
$$
for a quadratic form $f(x) = (A x,
x)$, $A=A^{T}$, with different eigenvalues of the matrix $A$
($\lambda_1 < \lambda_2 < ... < \lambda_n$),
on the unit sphere $\{x \in \R^n : \|x\| = 1 \}$.
\end{lemma}
The proof can be found in section~\ref{sec:wEB-on-sphere-proof} of the
Appendix.
\end{example}

\begin{example}
\rm  A quadratic form satisfies the tEB condition on the Stiefel manifold
$S_{n, k} = \{X \in \R^{n \times k} : X^T X = I_k \}$  \cite[Corollary 1]{liu-etal2015}.
This fact generalizes Lemma \ref{lem:wEB-on-sphere}.
\end{example}

\rm If problem (\ref{eq:main}) is nondegenerate then
the tEB condition holds.
\begin{theorem}
\label{thm:wEB-for-nondegenerate}
If problem $\eqref{eq:main}$ is nondegenerate, then the tEB condition is fulfilled with some constant $\mu>0$.
\end{theorem}
The proof can be found in section~\ref{sec:wEB-for-nondegenerate-proof}
of the Appendix.

If we demand additionally thrice
continuous differentiability of all functions
then, by Taylor's formula with  the Lagrange form of the remainder,
one can estimate radius $r$ and constant $\mu$ in the tEB condition in Formula
(\ref{eq:In}) via Lipschitz constants.

We want to pay attention that  the gEB condition holds under conditions of Theorem~\ref{thm:wEB-for-nondegenerate} by
Proposition~\ref{cor:gradient-mapping-greater-F'}.
Suppose that $\Omega$ is the set of global minima
from the set  $S\cap\{ x\ :\ f(x)\le f(x_0)\}$, and
$x_0\in S$ is a starting point for the GPA.
From  \cite{MSbReview}
we get that the GPA converges to some element of the set  $\Omega$ with
linear rate. Also in the case when the tEB condition is valid, we immediately get
linear convergence of GPA, see \cite[Section 3.4]{bpt}.

Finally note that condition of non-degeneracy for problem
$\eqref{eq:main}$ is not necessary for fulfillment of the tEB
condition. Assume that the function $x\to \lambda_{x}$ is
continuously differentiable in a neighborhood of the set
 $S$ of the form $U_{S}(\delta)$, $\delta>0$, and there exists a number
 $\mu>0$ such that for any point $x_{*}\in \Omega$
the next condition holds
$$
\left\|F'_{z}(x_{*},\lambda_{x_{*}})\left(\begin{array}{c} I_{n} \\
\lambda_{x}'|_{x=x_{*}}
\end{array}\right)h\right\| \ge\mu \| h\|,\quad \forall
h\in T_{x_{*}}\subset\R^n.
$$
Then the tEB condition holds in problem (\ref{eq:main}). The
proof repeats the proof of
Theorem~\ref{thm:wEB-for-nondegenerate}. We shall not discuss
this approach for proving the tEB condition because we need
existence of the inverse matrix  $F'(x,\lambda_{x})$, for any
$x\in \Omega$, in our situation.

\subsection{Convergence of the Algorithm}
\label{sec:main-results-part-2}

\begin{lemma}
\label{lem:sigma-around-stationary} Suppose that $f, g_i \in
\mathcal{C}^2$, the function $\lambda_x$ is defined in
\eqref{eq:lambda_x-definition}, the function $F'(x, \lambda_x)$,
is defined in \eqref{eq:F'_x-definition} and is
Lipschitz continuous on  $S$ with constant $L_{1, Fx}$%
\footnote{It is sufficient to demand Lipschitz continuity of
$F'(x, \lambda_x)$ in some neighborhood of stationary points.
This leads to one more restriction of the value  $\beta$ from
above.
%радиуса $\beta/(\sigma_0 L_{1,Fx})$.
Moreover, we can consider weaker condition $\|F'(x, \lambda_x) -
F'(x_*, \lambda_{x_*})\| \leq L_{1,Fx} \|x - x_*\|$,
where $x_*$ is the nearest stationary point to the point $x\in S$.}%
: $\|F'(x, \lambda_x) - F'(y, \lambda_{y})\| \leq L_{1, Fx} \|x - y\|$
forall  $x, y \in S$.

Then for any  $\beta \in [0, 1)$ the next estimate
$$ %  \beta === L_3 \sigma_0 r
\|F'(x, \lambda_x)\| \leq \frac{\sigma_0}{1 - \beta}, \quad
\forall x \in S : \;\rho(x, \Omega) \leq \frac{\beta}{ \sigma_0
L_{1,Fx}}
$$
takes place.
\end{lemma}

The proof can be found in
section~\ref{sec:sigma-around-stationary-proof} of the Appendix.

Gathered together the above mentioned results we obtain the
theorem about the convergence of the considered combined algorithm. Recall that
$\gamma$ is a fix step-size in the GPA (1st phase) and $\Delta f
= f(x_0) - \min_{x \in S} f(x)$ is the fluctuation of the
function.
\begin{theorem}
\label{thm:convergence-to-stationary} Let $\Omega =
\{x_{j}\}_{j=1}^{J}$ be the set of stationary points in problem
\eqref{eq:main}, $\Sigma = \{ [x,\lambda_{x}]\ :\ x\in\Omega\}$.
Assume that problem \eqref{eq:main} is nondegenerate  and the tEB
condition holds with constant $\mu$ \eqref{eq:weak-error-bound}.
 Suppose that in $d$-neigh\-bor\-hood $U_{\Omega}(d)$ of the set
$\Omega$ the function $S\ni x\to \lambda_x$ is Lipschitz
continuous with constant $L_{\lambda}$ and the function $S\ni
x\to F'(x, \lambda_x)$ is Lipschitz continuous with constant
$L_{1,Fx}$. Suppose that the function $F'(z)$ is Lipschitz
continuous with constant $L_{1,F}$ on the set $U_{\Sigma}(2r)$
for $r=d\sqrt{1+L_{\lambda}^{2}}$.
 Let $\sigma_0>0$ and $\beta\in (0,1)$ be such constants that estimate
 \eqref{eq:sigma-0-definition} holds,
$\beta \leq L_{1,Fx} \sigma_0 d$ and
\begin{equation}\label{bsl}
\frac{1-\beta}{2\sigma_{0}L_{1,F}}\le r.
\end{equation}
Then for any point $x_0\in S$ the Algorithm with the switching
condition
\begin{equation} \label{eq:switching-condition-with-constants}
\|P_{T_x} f'(x)\| \leq C = \min\Big\{\frac{\mu \beta}{L_{1,Fx}
\sigma_0}, \; \frac{(1 - \beta)^2}{4 L_{1,F} \sigma_0^2} \Big\},
\end{equation}
or with another condition
\eqref{eq:switching-condition-on-delta-x} $\|x_k - x_{k-1}\| \leq
\frac{\gamma}{1 + \gamma L_1} C$, converges to some stationary
point  $x_* \in \Omega$. We need no more than
$$
N(\varepsilon)
= N_1(C) + N_2(\varepsilon)
= \left\lceil \frac{2 \Delta f (1 + \gamma L_1)^2}%
{C^2 \gamma (1 - \gamma L_1)}%
\right\rceil
+ \Big\lceil \log_2 \Big(\frac{C \sigma_0}{\varepsilon (1 - \beta)} \Big) \Big\rceil + 1
$$
steps of the Algorithm to achieve the inequality $\rho(x_k, \Omega)
\leq \varepsilon$. In particular, we need  $N_1(C)$ steps of the
GPA and  $N_2(\varepsilon)$  steps of the modified NM.
\end{theorem}

The proof can be found in section
~\ref{sec:main-algorithm-convergence-proof}.

Note that the first phase of the Algorithm, the GPA, is more difficult
from computational point of view. One can minimize the
\emph{number of steps} for the GPA, choosing the step-size
$\gamma$ and minimizing  $N_1(C)$ as function of $\gamma$. The
optimal value is
$$
\gamma^* = \min\Big\{\frac{1}{3L_1}, \frac{R}{L_0}\Big\}.
$$

\section{Acknowledgements}
The work was supported by Russian Science Foundation (Project
16-11-10015).

The authors are greatful to B. T. Polyak for useful comments and
suggestions.

\section{Appendix: proofs.}

\subsection{Proof of Proposition~\ref{prop:stiefel-is-proximally-smooth}}
\label{sec:projection-on-simple-sets}

\begin{proposition}\label{prop:stiefel-is-proximally-smooth}
    The Stiefel manifold $S_{n,k}$ is proximally smooth set with constant $R=1$
    for all $n,k$, $n\ge k$.

\end{proposition}

\begin{proof}
Proximal smoothness of the Stiefel manifold $\mathcal{S}_{n, k} =
\{X \in \R^{n \times k} : X^T X = I_k\}$, $ k \leq n$, with
constant of proximal smoothness $R=1$ follows from the result
about explicit form for the metric projection onto the Stiefel
manifold \cite[Proposition 7]{absil-malick2012}:
$$
P_{\mathcal{S}_{n, k}} (X) = U I_{k, n} V^T, \;\; \forall X = U
\Sigma V^T: \;\; \rho(X, \mathcal{S}_{n, k}) < 1.
$$
For the set  $\{X \in \R^{n \times k}: \rho(X, \mathcal{S}_{n,
k}) < 1\}$ the metric projection is a singleton. Here $U, V$ are
orthogonal matrices of a singular values decomposition for matrix
$X$, $I_{k, n} = [I_k, 0] \in \R^{n \times k}$. The distance
between matrices is understood in the Frobenius metric $\rho(X,
Y) = \| X-Y\|=\left( \mbox{\rm trace\,
}(X-Y)^{T}(X-Y)\right)^{1/2}$.

From upper smicontinuity of the metric projection in finite
dimensional space $R^{n \times k}$ \cite[Ch. 3, \S 1, Proposition
23]{Aubin} and its uniqueness we obtain that the function $U_{\mathcal{S}_{n,
k}}(1)\ni X \rightarrow P_{\mathcal{S}_{n, k}}X$ is continuous.

Finally, constant of proximal smoothness can not exceed 1. For
the point (matrix) $X_0 = [0, e_2, e_3, ..., e_k]\in\R^{n\times
k}$ with $\rho(X_0, \mathcal{S}_{n, k}) = 1$ there exists at
least two metric projections onto the Stiefel manifold: $X_{0,-} =
[-e_1, e_2, e_3, ..., e_k]$ and $X_{0,+} = [e_1, e_2, e_3, ...,
e_k]$.
\end{proof}

\subsection{Proof of Theorem~\ref{thm:projected-gradient-convergence}, Corollary~\ref{cor:projected-gradient-number-of-steps}
and Proposition~\ref{cor:gradient-mapping-greater-F'}}
\label{sec:projected-gradient-convergence-proof}

\begin{proof}[Proof of Theorem~\ref{thm:projected-gradient-convergence}]

For a natural $k$ define the functions
$$
\psi_k(x) = f(x_k) + (f'(x_k), x - x_k) + \frac{1}{2 \gamma} \|x
- x_k\|^2.
$$
By Lipschitz continuity of the gradient  $f'$ we have
$$
f(x_k) + (f'(x_k), x - x_k) + \frac{L_1}{2} \|x - x_k\|^2\ge
f(x),\quad \forall x\in \R^{n}\qquad\mbox{\rm and}
$$
\begin{equation} \label{eq:psi_k-as-upper-bound}
\psi_k(x) \geq f(x) + \frac{1}{2}\Big(\frac{1}{\gamma} - L_1\Big) \|x - x_k\|^2,
\end{equation}
for all $x \in \R^n$.

By the conditions $\gamma < \frac{R}{L_0}$,  $x_k \in S$, we get
$x_{k} - \gamma f'(x_k)\in S \cup U_{S}(R)$. Indeed
$$
\rho\left(x_k - \gamma f'(x_k), S \right)
\leq \gamma \left\| f'(x_k) \right\|
< \frac{R}{L_0} L_0 = R.
$$
Hence $x_{k+1}$ is a unique metric projection of the point $x_{k}
- \gamma f'(x_k)$ onto the set $S$.

From the equality
$$
x_{k+1} = P_{S}(x_k - \gamma f'(x_k))
= \arg \min_{x \in S} \|x - (x_k - \gamma f'(x_k))\|^2
= \arg \min_{x \in S} \psi_k(x)
$$
it follows that $\psi_k(x_{k+1}) \leq \psi_k(x)$ for all $x \in
S$. From the same equality and necessary condition of extremum
for function $\psi_k$ on the set  $S$ we have
\begin{equation} \label{eq:necessary-extremal-condition-by-normal-cone}
\psi_k'(x_{k+1}) = f'(x_k) + \frac{1}{\gamma} (x_{k+1} - x_k) \in
-\mathcal{N}(S, x_{k+1}), \;\; k \geq 0.
\end{equation}

Thus, taking into account inequality
\eqref{eq:psi_k-as-upper-bound}, we obtain
\begin{equation} \label{eq:decrease-of-f_x_k}
f(x_k)
= \psi_k (x_k)
\geq \psi_k (x_{k+1})
\geq f(x_{k+1}) + \frac{1}{2}\Big(\frac{1}{\gamma} - L_1\Big) \|x_{k+1} - x_k\|^2.
\end{equation}
Note that $\frac{1}{\gamma} - L_1>0$. From boundedness  of the
function $f$ on the compact set $S$  we get
\begin{equation} \label{eq:projected-gradient-convergence-of-delta-x_k}
\lim_{k \rightarrow \infty} \|x_{k+1} - x_k\| = 0.
\end{equation}

For proving $\lim\limits_{k\to\infty}\rho (x_{k},\Omega)=0$
suppose the contrary:  $\lim\limits_{k \rightarrow \infty}
\rho(x_k, \Omega) \neq 0$. Then there exists  $\epsilon
> 0$ and subsequence $x_{k_i}$ with $\rho(x_{k_i},
\Omega) \geq \epsilon$ for all $i$. Consider a converging
subsequence (that exists by compactness of the set $S$), denote
it again  $\{x_{k_i}\}$. Let $x_* = \lim\limits_{i \to \infty}
x_{k_i}$.

By \eqref{eq:necessary-extremal-condition-by-normal-cone} we have
$$
f'(x_{k_i}) + \frac{1}{\gamma} (x_{k_i + 1} - x_{k_i})
\in -\mathcal{N}(S, x_{k_i + 1}).
$$
Passing to the limit $i \to \infty$, using upper semicontinuity
of the normal cone  $\mathcal{N}(S, \cdot)$ and the limit property
\eqref{eq:projected-gradient-convergence-of-delta-x_k} we obtain
$$
f'(x_*) \in -\mathcal{N}(S, x_{*}).
$$
But $\lim\limits_{i \to \infty} \rho(x_{k_i}, \Omega) = 0 \ge
\epsilon$. A contradiction. This implies the convergence of the
sequence $\{x_{k_i}\}$ to the set of stationary points.

Inclusion \eqref{eq:necessary-extremal-condition-by-normal-cone}
implies the next important inequality
\begin{equation} \label{eq:relation-betweed-rho-and-delta-x}
\| P_{T_{x_{k+1}}}f'(x_{k+1})\|=\rho\big(-f'(x_{k+1}),
\mathcal{N}(S, x_{k+1})\big) \leq \Big(\frac{1}{\gamma} +
L_1\Big) \|x_{k+1} - x_k\|,
\end{equation}
we use here Lipschitz continuity of gradient $\|f'(x_{k+1}) -
f'(x_k)\| \leq L_1 \|x_{k+1} - x_k\|$. By formula
\eqref{eq:projected-gradient-convergence-of-delta-x_k} it follows
that  $\rho\big(-f'(x_k), \mathcal{N}(S, x_k)\big) \to_{k \to
\infty} 0$.
\end{proof}

\begin{proof}[Proof of Corollary~\ref{cor:projected-gradient-number-of-steps}]
The value $\Delta f = f(x_0) - f_{\min}$ is bounded by the
Weierstrass theorem. Consistently using inequality
\eqref{eq:decrease-of-f_x_k} $N$ times we get
$$
\Delta f
= f(x_0) - f_{\min} \geq f(x_0) - f(x_{N})
\geq \frac{N}{2}\Big(\frac{1}{\gamma} - L_1\Big) \min_{i=1,..., N}\|x_{i} - x_{i-1}\|^2.
$$
Thus there exists  $1 \leq i \leq N$ such that
\begin{equation} \label{eq:delta-x-on-N}
\|x_{i} - x_{i-1}\|^2 \leq \frac{2 \Delta f}{N \cdot
(\frac{1}{\gamma} - L_1)}
\end{equation}
and from \eqref{eq:relation-betweed-rho-and-delta-x} for given
$i$ the next estimate holds
$$
\rho\big(-f'(x_{i}), \mathcal{N}(S, x_{i})\big) \leq
\Big(\frac{1}{\gamma} + L_1\Big) \sqrt{\frac{2 \Delta f}{N \cdot
(\frac{1}{\gamma} - L_1)}}.
$$
Hence for an arbitrary accuracy $\varepsilon
> 0$ in no more than
$$
N(\varepsilon)
= \left\lceil \frac{2 \Delta f (\frac{1}{\gamma} + L_1)^2}%
{\varepsilon^2  (\frac{1}{\gamma} - L_1)}%
\right\rceil
= \left\lceil \frac{2 \Delta f (1 + \gamma L_1)^2}%
{\varepsilon^2 \gamma (1 - \gamma L_1)}%
\right\rceil
$$
steps we shall find a point  $x_i$ with $ \|
P_{T_{x_{i}}}f'(x_{i})\| =\rho\big(-f'(x_{i}), \mathcal{N}(S,
x_{i})\big) \leq \varepsilon$.

Finally, this point and the number of iteration  $i$ also can be
found from condition \eqref{eq:relation-betweed-rho-and-delta-x}:
\begin{equation}\label{eq:est-on-xi}
 \|x_{i} - x_{i-1}\| \leq
\frac{\varepsilon}{\frac{1}{\gamma} + L_1} = \frac{\varepsilon}{1
+ \gamma L_1} \gamma.
\end{equation}
\end{proof}

\begin{proposition}
\label{cor:gradient-mapping-greater-F'} Suppose that conditions
of Theorem~\ref{thm:projected-gradient-convergence} hold and the
set $S$ is a compact, smooth and proximally smooth with constant
$R$ manifold without edge, that is given by the system
\eqref{eq:S-definition}. Assume that the tEB condition holds with
constant $\mu$. Then for all $x \in S$ the gEB condition holds
\begin{gather*}
\frac{\mu}{1+L_{1}\gamma_{0}+\mu\gamma_{0}}\rho (x,\Omega)\le
\|\g_{\gamma}(x)\|, \quad\mbox{\rm where}\ \gamma \in (0,
\gamma_{0}),\quad \gamma_{0}=\min\Big\{\frac{1}{L_1},
\frac{R}{L_0}\Big\}.
\end{gather*}
\end{proposition}

\begin{proof}[Proof of Proposition~\ref{cor:gradient-mapping-greater-F'}]
Fix $x\in S$. Let $x_{1}=P_{S}(x-\gamma f'(x))$ and $
\g_{\gamma}(x) = \frac{1}{\gamma}\| x-x_{1}\|$. Then we obtain by
formula (\ref{eq:relation-betweed-rho-and-delta-x}) with $x_{k}=x$
and $x_{k+1}=x_{1}$ that
$$
\| P_{T_{x_{1}}}f'(x_{1})\|\le \left(
L_{1}+\frac{1}{\gamma}\right)\| x-x_{1}\|.
$$
From 1-Lipschitz condition for the distance function and the
 tEB condition we have
$$
\rho (x,\Omega)-\| x-x_{1}\|\le \rho (x_{1},\Omega)\le
\frac{1}{\mu}\| P_{T_{x_{1}}}f'(x_{1})\|\le
\frac{L_{1}+\frac{1}{\gamma}}{\mu}\| x-x_{1}\|=
\frac{1+L_{1}\gamma}{\mu}\| \g_{\gamma}(x)\|,
$$
$$
\rho (x,\Omega)\le \frac{1+L_{1}\gamma}{\mu}\|
\g_{\gamma}(x)\|+\gamma\frac{\|
x-x_{1}\|}{\gamma}\le\frac{1+L_{1}\gamma_{0}+\mu\gamma_{0}}{\mu}\|
\g_{\gamma}(x)\|.
$$
\end{proof}

\subsection{Proof of Lemma~\ref{lem:finite-number-of-stationary-points-in-nondegenerate}}
\label{sec:finite-number-of-stationary-points-in-nondegenerate}

\begin{proof}
Suppose the contrary: the set $\Omega$ of stationary points in
nondegenerate problem is infinite. From compactness of the set $S$
the set $\Omega$ has a limit point: $x_{i}\to x_{*}$,
$x_{i},x_{*}\in\Omega$. By continuity of the function $S\ni x\to
\lambda_{x}$ we have $z_{i}=[x_{i},\lambda_{x_{i}}]\to z_{*} =
[x_{*},\lambda_{x_{*}}]$.

Consider the Taylor formula for the function $F(z)$ in a
neighborhood of the point $z_* = (x_{*},\lambda_{x_{*}})$:
$$
F(z_i) - F(z_*) = F'(z_{*}) (z_i - z_*) + o(\|z_i - z_*\|), \;\;
i = 0, 1, ...
$$
Points  $x_i$ and $x_*$  are stationary, then
$F(z_{i})=F(z_{*})=0$ and the next inequality holds
$$
F'(z_*) \frac{z_* - z_i}{\|z_* - z_i\|} = \frac{o(\|z_i -
z_*\|)}{\|z_i - z_*\|}.
$$
Vectors $\frac{z_* - z_i}{\|z_* - z_i\|}$ have unit length and
without loss of generality converge to a vector $u_*$, $\|
u_{*}\|=1$. Passing to the limit $i\to\infty$ we get
 $F'(z_*) u_* = 0$ for the unit vector $u_* \in \R^{n + m}$. The matrix
 $F'(z_{*})$ is
degenerate. A contradiction.
\end{proof}

\subsection{Proof of Lemma~\ref{lem:wEB-on-sphere}}
\label{sec:wEB-on-sphere-proof}

\begin{proof}
The eigenvalues of the symmetric matrix $A \in \R^{n \times n}$
are different. Hence there are  $n$ different unit eigenvectors.
In the basis from these vectors we have $f(x) = \sum_{j=1}^{n}
\lambda_k x_k^2$. There are  $2n$ stationary points $\pm e_j$.

Put $\mu = \min_{i \neq j} |\lambda_i - \lambda_j| > 0$.

For a point $x$ on the unit sphere define the nearest stationary
point (arbitrarily, if there are several of them). This nearest
point corresponding to the $k$th eigenvector (with particular
sign),which is chosen from the condition
$k_x = \arg\min_{i=1,...,n}\|x \pm e_i\|$. %= \arg \max_{i = 1,..., n} |x_i|$.
Further for a fixed point $x$ denote this number without index,
i.e. $k \doteq k_x$ and $e_{k_x}=e_k$. It is obviuos that the
point $x$ belongs to the spherical segment ("hat" of the sphere),
say
$$
S_{k} = \left\{x \in \R^n : \|x\| = 1, x_{k} \geq \frac{1}{\sqrt{2}} \right\},
$$
$k$ corresponds to the nearest eigenvector.

Reorder coordinates in such way that component $x_{k}$ will be
the last:  $x = [u, x_k]$, here $u = [x_1, ..., x_{k-1}, x_{k+1},
..., x_n] \in \R^{n-1}$. If the point $x$ belongs to the segment
$S_{k}$, then $k$th coordinate can be expressed through the rest
coordinates, i.e. through $u$. Denote $h(u) = f([u, \sqrt{1
- \|u\|^2}])$, $h(u)= f(x)$. Then
$$
h(u) = \sum_{i=1}^n (\lambda_i - \lambda_k) x_i^2, \;\;
\|h'(u)\|^2 = \sum_{i=1}^n (\lambda_i - \lambda_k)^2 x_i \geq 4
\mu^2 \|u\|^2,
$$
 $\|h'(u)\| \geq 2 \mu \|u\|$ and consequently
\begin{equation} \label{eq:u-upper-bound}
\|u\| \leq \frac{1}{2 \mu} \|h'(u)\|.
\end{equation}

Define the subspace $H_k = \{x \in \R^n : (x, e_k) = 0\}$ (i.e.
$x_k = 0$).
Note that  $u$ can be expressed as  $u = P_{H_k} x$. Then we can
estimate the value $\|u\|$ \emph{from below} via the distance
between $x$ and the nearest orth $e_{k_x}=e_k$. If $x \in S_k$
then the angle between the segments with endpoints $0$, $u$ and $e_k$, $x$ no
more than $\pi/4$, see Fig.
\ref{fig:connection-between-u-and-x_k}. Hence
\begin{equation} \label{eq:u-lower-bound}
\frac{\|x - e_{k}\|}{\sqrt{2}}
\leq \|u - 0\| = \|u\|.
\end{equation}
\begin{figure}[!h]
\centering
\includegraphics[width=17cm]{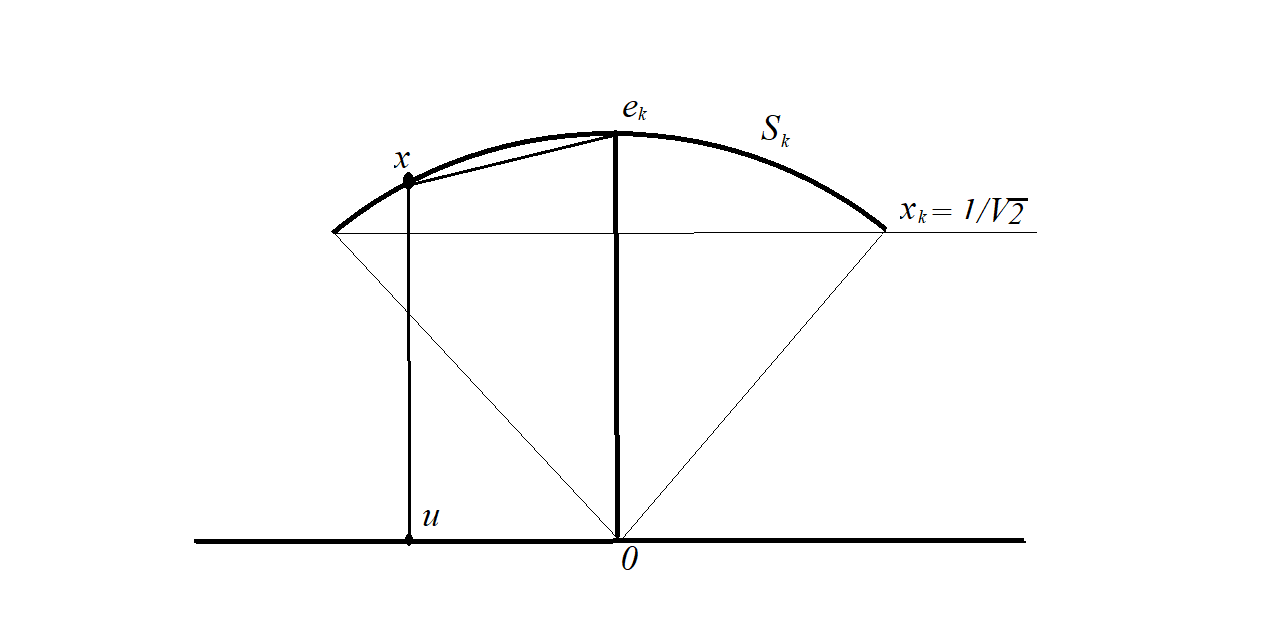}
\caption{Relationship between  $\|u\|$ and $\|x - e_k\|$.}
\label{fig:connection-between-u-and-x_k}
\end{figure}

Now we should estimate the norm of the vector $h'(u)$ through the
norm $\| P_{T_x} f'(x)\|$ from above. Define the function $\phi(u)
= [u, \sqrt{1 - \|u\|^2}]$, "restoring" a vector $x\in S_{k}$ by
its part $u$.

Firstly,  $h(u) = f(\phi(u))$, thus $h'(u) = f'(x) \phi'(u)$, here
$\phi'$ is the Jacobi matrix for the function $\phi(\cdot)$:
\begin{gather*}
\phi'(u) = \begin{bmatrix}
1 & 0 & \cdots & 0 & 0 \\
0 & 1 & \cdots & 0 & 0 \\
\vdots & \vdots & \ddots & \vdots & \vdots \\
0 & 0 & \cdots & 1 & 0 \\
0 & 0 & \cdots & 0 & 1 \\
\frac{- u_1}{\sqrt{1 - \|u\|^2}} &
\frac{- u_2}{\sqrt{1 - \|u\|^2}} &
\cdots &
\frac{- u_{n-2}}{\sqrt{1 - \|u\|^2}} &
\frac{- u_{n-1}}{\sqrt{1 - \|u\|^2}} &
\end{bmatrix} \\
=
\begin{bmatrix}
1 & 0 & \cdots & 0 & 0 \\
\vdots & \vdots & \ddots & \vdots & \vdots \\
0 & 0 & \cdots & 0 & 1 \\
\frac{- x_1}{\sqrt{1 - \|u\|^2}} &
\frac{- x_2}{\sqrt{1 - \|u\|^2}} &
\cdots &
\frac{- x_{n-1}}{\sqrt{1 - \|u\|^2}} &
\frac{- x_{n}}{\sqrt{1 - \|u\|^2}} &
\end{bmatrix}
\in \R^{n \times (n-1)}
\end{gather*}
(there is no column with $x_k$ in the last matrix). Let $\ell\in
\R^{m}$ be a unit vector. We shall estimate the norm of
$\phi'(u)$:
\begin{gather*}
\|\phi'(u) \ell \|^2 = \sum_{j=1}^{n-1} \ell_j^2 +
\Bigg(\sum_{\scriptsize \begin{matrix}i=1, n \\ i \neq
k\end{matrix}} \frac{x_i \ell_{i_+}}{\sqrt{1 - \|u\|^2}}\Bigg)^2
\leq 1 + \frac{\|u\|^2}{1 - \|u\|^2} \leq 2.
\end{gather*}
We use the inclusion  $x \in S_k$ in the last inequality that
means  $\|u\| \leq 1/\sqrt{2}$ and a technical index
\[
i_+ = \left\{
\begin{matrix}
i, & \text{ если } i < k,\\
i-1, & \text{ если } i > k.
\end{matrix}
\right.
\]
Finally we have $\|\phi'(u) \ell \| \leq \sqrt{2}$.

Substitute a unit vector $\ell(u)$ with $\|h'(u)\| = (h'(u),
\ell(u))$. Then
\begin{gather*}
\|h'(u)\|
= |(h'(u), \ell(u))|
= |(\phi'(u)^T f'(\phi(u)), \ell(u))|
= |(f'(\phi(u)), \phi'(u) \ell(u))| \\
= \|f'(\phi(u))\|\cdot\|\phi'(u) \ell(u)\|\cdot \cos \beta
\leq \sqrt{2} \|f'(x)\|\cdot \cos \beta,
\end{gather*}
here $\beta \leq \pi/2$ is the angle between \emph{directions} of
 vectors $f'(x)$ and $\phi'(u) \ell(u)$. Denote by $\gamma$ the
angle between $f'(x)$ and the tangent subspace  $T_x$. We have
$\phi'(u) \ell(u) \in T_x$, hence $\gamma \leq \beta$ and
\begin{equation} \label{eq:h_u-upper-bound}
\|h'(u)\|
\leq \sqrt{2} \|f'(x)\| \cos \beta
\leq \sqrt{2} \|f'(x)\| \cos \gamma
= \sqrt{2} \|P_{T_x} f'(x)\|.
\end{equation}
By inequalities \eqref{eq:u-upper-bound},
\eqref{eq:u-lower-bound} and \eqref{eq:h_u-upper-bound} we get
$$
\frac{\|x - e_{k}\|}{\sqrt{2}} \leq \|u\| \leq \frac{1}{2 \mu}
\|h'(u)\| \leq \frac{1}{\sqrt{2} \mu} \|P_{T_x} f'(x)\|.
$$
So, $\rho(x, \Omega) \leq \frac{1}{\mu} \|P_{T_x} f'(x)\|$.
\end{proof}

\subsection{Proof of Theorem~\ref{thm:wEB-for-nondegenerate}}
\label{sec:wEB-for-nondegenerate-proof}

In the present section a lower index means the number of the point
in a finite set, but not a number of iteration.

\begin{proof}
By
Lemma~\ref{lem:finite-number-of-stationary-points-in-nondegenerate}
 the set of stationary points in nondegenerate problem is finite
($\Omega = \{x_{j}\}_{j=1}^{J}$).

The set  $S$ is
compact $\mathcal{C}^2$-smooth manifold and $f\in \mathcal{C}^2$, hence the function
$$
S\ni x\to\lambda_x = - (g'(x) g'(x)^T)^{-1}g'(x)f'(x)
$$
is Lipschitz continuous
with some constant $L_\lambda$.

From the definition of stationary points  $F(x_{j},
\lambda_{x_{j}}) = 0$. Then from differentiability of  $F(z)$ by
the Taylor formula we have for any $j\in \{1,\dots,J\}$

\begin{gather*}
%F(x, \lambda_x) - F(x_{j}, \lambda_{x_{j}})=
F_x(x) \doteq F(x, \lambda_x) = F(x, \lambda_x) - F(x_{j}, \lambda_{x_{j}})
= F'(x_{j}, \lambda_{x_{j}})
\begin{bmatrix}
x - x_{j} \\
\lambda_x - \lambda_{x_{j}}
\end{bmatrix}
+ o_j(\varrho),\ \varrho\to +0,
\end{gather*}
here $\varrho = \sqrt{\|x - x_{j}\|^2 + \|\lambda_x -
\lambda_{x_{j}}\|^2} \leq \|x - x_{j}\| \sqrt{1 + L_\lambda^2}$.

By the inverse operator theorem condition of non-degeneracy of
the problem $\|F'(x_{j}, \lambda_{x_{j}})^{-1}\| \leq \sigma_0$,
$j = 1,...,J$, is equivalent to the condition
$$
\|F'(x_{j}, \lambda_{x_{j}}) h \|
\geq \frac{1}{\sigma_0} \|h\|,
\;\; \forall h \in \R^{n + m}, \; j = 1, ..., J.
$$
Choose such a number $\ell>0$ that for all $j = 1,..., J$ and $h
\in \R^{n+m}, \;\|h\| \leq \ell$, the estimate $\|o_j(\|h\|)\|
\leq \frac{1}{2\sigma_0} \|h\|$ is fulfilled.

Fix $x \in S$,  $\rho(x, \Omega) \leq \frac{\ell}{\sqrt{1 +
L_\lambda^2}}$. Using the Taylor expansion with respect to the
\emph{nearest} point  $x_{j}\in\Omega$ at the point $x$ and taking
in mind that $\varrho \leq \|x - x_{j}\| \sqrt{1 + L_\lambda^2}
\leq \ell$ we have
\begin{gather*}
\|F_x(x)\|
\geq \left\|F'(x_{j}, \lambda_{x_{j}})
\begin{bmatrix}
x - x_{j} \\
\lambda_x - \lambda_{x_{j}}
\end{bmatrix}\,
\right\|
- \|o_j(\varrho)\|
\geq \frac{1}{\sigma_0} \varrho - \frac{1}{2\sigma_0} \varrho
= \frac{1}{2\sigma_0} \varrho \\
\geq \frac{1}{2\sigma_0} \|x - x_{j}\|
= \frac{1}{2\sigma_0} \rho(x, \Omega).
\end{gather*}
Thus the tEB condition holds
\begin{equation}\label{eq:In}
\|P_{T_x} f'(x)\| = \|F_x(x)\| \geq \mu \rho(x, \Omega),\quad
\forall x\in \cup_{j=1}^{J}\int B_{r}(x_{j}),
\end{equation}
where $\mu=1/(2\sigma_{0})$,
$r=\frac{\ell}{\sqrt{1+L_{\lambda}^{2}}}$.

The function $F_x(x)$ is continuous on the compact set
$S_{1}=\left\{x \in S : \rho(x, \Omega) \ge \frac{\ell}{\sqrt{1 +
L_\lambda^2}} \right\}$. From $S_{1}\cap\Omega=\emptyset$ we get
$\|F_x(x)\| > 0$ for all $x\in S_{1}$. By the Weierstrass theorem
there exists a number $b > 0$ such that $\|F_x(x)\| \geq b
> 0$ for all $x\in S_{1}$. Since  $\diam S = \sup_{x, y \in
S}\|x - y\| \geq \rho(x, \Omega)$, then
\begin{equation}\label{eq:Out}
\|F_x(x)\| \geq b =
\frac{b}{\diam S} \diam S \geq \frac{b}{\diam S} \rho(x,
\Omega),\;\; \forall x \in S : \rho(x, \Omega) >
\frac{\ell}{\sqrt{1 + L_\lambda^2}}.
\end{equation}
Combining inequalities (\ref{eq:In}) and (\ref{eq:Out}), we prove
the tEB condition on the set $S$:
$$
\|F_x(x)\|
\geq \min\Big\{\frac{1}{2 \sigma_0}, \frac{b}{\diam S} \Big\} \rho(x, \Omega), \;\;
\forall x \in S.
$$
\end{proof}

\subsection{Proof of Lemma~\ref{lem:sigma-around-stationary}}
\label{sec:sigma-around-stationary-proof}

The proof is based on the estimate $\|(I + X)^{-1}\| \leq 1 / (1 -
\|X\|)$,  $\|X\| < 1$.
\begin{proof}
Suppose that a point $x$ is at a distance no more than $r = \beta
/ (\sigma_0 L_{1,Fx})$ from some stationary point $x_*$, say, the
nearest. Denote
$$
\Delta F' = \Delta F' (x) = F'(x, \lambda_x) - F'(x_*, \lambda_{x_*}).
$$
Fix $\beta\in (0,1)$. We have $\|F'(x, \lambda_x)^{-1}\|=$
\begin{gather*}
= \|(F'(x_*, \lambda_{x_*}) + \Delta F')^{-1}\|
\leq \|F'(x_*, \lambda_{x_*})^{-1}\|\cdot \left\| \left(I+F'(x_*, \lambda_{x_*})^{-1}\Delta F'\right)^{-1}\right\|
\le \\
\le \frac{\sigma_0}{1-\| F'(x_*, \lambda_{x_*})^{-1}\Delta F'\|}\le
\frac{\sigma_0}{1-\| F'(x_*, \lambda_{x_*})^{-1}\|\cdot \|\Delta F'\|}\le
\frac{\sigma_0}{1-\sigma_0 \|\Delta F'\|}.
\end{gather*}
Taking into account the estimate $\|\Delta F'\|\le L_{1,Fx}\|
x-x_*\|\le L_{1,Fx}r\le \frac{\beta}{\sigma_0}$ and the last
formula we obtain the statement of Lemma.
\end{proof}

\subsection{Proof of Theorem~\ref{thm:convergence-to-stationary}}
\label{sec:main-algorithm-convergence-proof} Recall the
expression for constant $C$:
$$C = \min\left\{\frac{\mu \beta}{L_{1,Fx}
\sigma_0}, \; \frac{(1 - \beta)^2}{4 L_{1,F} \sigma_0^2}
\right\}.$$

\begin{proof}
%First of all note that condition (\ref{bsl}) is not contradictory.
%The function $F'$ is Lipschitz continuous on any compact domain of $\R^{n+m}$.
%If (\ref{bsl}) is false then we can increase the value
%$\sigma_{0}$ and get the validity of (\ref{bsl}).

By Corollary \ref{cor:projected-gradient-number-of-steps} the GPA
in no
more than  $N_1(C) = \left\lceil \frac{2 \Delta f (1 + \gamma L_1)^2}%
{C^2 \gamma (1 - \gamma L_1)}%
\right\rceil$ steps will achieve a point  $\widehat{x}\in S$,
where condition \eqref{eq:switching-condition} (and
\eqref{eq:switching-condition-on-delta-x}) holds. Moreover, some
conditions will be met simultaneously.

Firstly, by $\|P_{T_{\widehat{x}}} f'(\widehat{x})\| \leq
\frac{\mu \beta}{\sigma_0L_{1,Fx}}$ and the tEB  condition we
obtain that the point $\widehat{x}$ is close to some stationary
point, i.e. $\rho(\widehat{x}, \Omega) \leq \frac{\|P_{T_x}
f'(\widehat{x})\|}{\mu} \leq \frac{\beta}{\sigma_0L_{1,Fx}}$.
From condition of Theorem  $\beta \leq L_{1,Fx} \sigma_0 d$ it
follows that $\rho(\widehat{x}, \Omega)<d$ and condition of
Lemma~\ref{lem:sigma-around-stationary} is fulfilled. Thus we
have the estimate $\|
F'(\widehat{x},\lambda_{\widehat{x}})^{-1}\|\le
\frac{\sigma_0}{1-\beta}$.

Secondly, by virtue of the choice of constant $C$
$$
\|F(\widehat{x}, \lambda_{\widehat{x}})\|
= \|P_{T_{\widehat{x}}} f'(\widehat{x})\|
\leq \frac{(1 - \beta)^2}{4 L_{1,F} \sigma_0^2}
%\leq \frac{1}{4 L_2 \sigma_1(r)^2}
\leq \frac{1}{4 L_{1, F} \|F'(\widehat{x},
\lambda_{\widehat{x}})^{-1}\|^2}
$$
and from (\ref{bsl}) we get
\begin{equation}\label{ball-r}
 2\|
F'(\widehat{x},\lambda_{\widehat{x}})^{-1}F(\widehat{x},\lambda_{\widehat{x}})\|\le
2\frac{\sigma_{0}}{1-\beta}\frac{(1-\beta)^{2}}{4L_{1,F}\sigma_{0}^{2}}
= \frac{1-\beta}{2L_{1,F}\sigma_{0}}\le r.
\end{equation}
Let $x_{j}\in \Omega $ be a nearest point to the point
$\widehat{x}$. From the equalities $d\sqrt{1+L_{\lambda}^{2}}=r$,
$\rho (\widehat{x},\Omega)=\| \widehat{x} - x_{j}\|<d$ and
Lipschitz condition for $\lambda_{x}$ we have
$$
\rho ([\widehat{x},\lambda_{\widehat{x}}],\Sigma)\le\|
[\widehat{x},\lambda_{\widehat{x}}] - [x_{j},\lambda_{x_{j}}]\|
\le \rho (\widehat{x},\Omega)\sqrt{1+L^{2}_{\lambda}}\le r
$$
 i.e. $F'(z)$ is
Lipschitz with constant $L_{1,F}$ on the ball
$B_{r}([\widehat{x},\lambda_{\widehat{x}}])$.

From the Lipschitz property for $F'(z)$ on the ball
$B_{r}([\widehat{x},\lambda_{\widehat{x}}])$ and formula
(\ref{ball-r}) by Proposition~\ref{prop:mod-newton-convergence}
the sequence $\widehat{x}_{k}$ converges  to a stationary point
$x_*\in\Omega$ with rate
$$
\|\widehat{x}_k - x_*\| \leq 2^{1-k} \|
F'(\widehat{x},\lambda_{\widehat{x}})^{-1}\| C \leq 2^{1-k}
\frac{C \sigma_0}{1-\beta}.
$$
For an arbitrary $\varepsilon > 0$ the number of steps of the
modified NM for the accuracy  $\rho(\widehat{x}_k, \Omega) \leq
\varepsilon$ can be estimated as follows
$$
N_2(\varepsilon)
= \Bigg\lceil \log_2 \Big(\frac{C \sigma_0}{\varepsilon (1 - \beta)} \Big) \Bigg\rceil + 1.
$$

\end{proof}


\begin{thebibliography}{99}

\bibitem{zabotin-chernyaev2001}
Yu. A. Chernyaev, An extension of the gradient projection method
and Newton's method to extremum problems constrained by a smooth
surface, Computational Mathematics and Mathematical Physics,
2015, 55:9, p.~1451-1460.

\bibitem{bpt}
M. Balashov, B. Polyak, A. Tremba, Gradient projection and
conditional gradient methods for constrained nonconvex
minimization. (2019) arXiv:1906.11580


\bibitem{Vial}
J.-P. Vial, Strong and Weak Convexity of Sets and Functions.
Mathematics of Operations Research, 1983, 8:2, pp.~231--259.

\bibitem{Clarke} F.\,H. Clarke, R.\,J. Stern, P.\,R. Wolenski, Proximal
smoothness and lower--$C^{2}$ property, J. Convex Anal.,
 2:1-2 (1995), 117--144.

\bibitem{karimi-etal2016}
H. Karimi, J. Nutini, M. Schmidt, Linear convergence of gradient
and proximal-gradient methods under the Polyak-Lojasiewicz
condition. In: Frasconi P., Landwehr N., Manco G., Vreeken J.
(eds) Machine Learning and Knowledge Discovery in Databases.
Lecture Notes in Computer Science, vol 9851. Springer, 2016.

\bibitem{drusvyatskiy-lewis2016}
D. Drusvyatskiy, A.S. Lewis, Error bounds, quadratic growth, and
linear convergence of proximal methods, (2016), arXiv:1602.06661

\bibitem{gao-etal2016}
B. Gao, X. Liu, X. Chen, Ya. Yuan, On the Lojasiewicz exponent of
the quadratic sphere constrained optimization problem,
(2016), arXiv:1611.08781v2

\bibitem{liu-etal2015}
H. Liu, W. Wu, A. M.-Ch. So, Quadratic Optimization with
Orthogonality Constraints: Explicit Lojasiewicz Exponent and
Linear Convergence of Line-Search Methods. (2015),
arXiv:1510.01025

\bibitem{absil-etal2008}
P.-A. Absil, R. Mahony, R. Sepulchre, Optimization Algorithms on
Matrix Manifolds, Princeton University Press, Princeton and
Oxford, 2008.

\bibitem{Aubin}
J.-P. Aubin, I. Ekeland, Applied Nonlinear Analysis. Wiley, 1984.

\bibitem{Bertsekas} D.P. Bertsekas, Nonlinear programming, Massachusetts, Athena Scientific, 2nd ed.,
1999. 777p.

\bibitem{Polyak} B. T. Polyak, Introduction to optimization, NY, Translation series of mathematics and engineering,
1987.

\bibitem{poliquin-etal2000}
R. A. Poliquin, R. T. Rockafellar, L. Thibault, Local
Differentiability of Distance Functions. Transactions of the
American Mathematical Society, 352:11 (Nov., 2000),
pp.~5231--5249.

\bibitem{Thibault} M. Bounkhel and L. Thibault, On various notions of regularity
of sets in nonsmooth analysis, Nonlin. Anal. 48 (2002), 223-246.

\bibitem{bookVol2} {M. V. Balashov} {Nonconvex optimization.} ---
In book: Control theory (additional chapters): Tutorial / Ed. D.
A. Novikov. Ч M.: Leland, 2019. Ч 552 p. In russian.

\bibitem{kolmogorov-fomin2004}
Kolmogorov A. N., Fomin S. V. Elements of the function theory and
functional analysis. -- 7th ed. -- M.: Fizmatlit, 2004. In
russian.

\bibitem{Nesterov}
Yu. Nesterov,
Lectures on Convex Optimization, Springer, 2018.

\bibitem{absil-malick2012}
P.-A. Absil, J. Malick, Projection-like retractions on matrix
manifolds. SIAM Journal on Optimization, Society for Industrial
and Applied Mathematics, 2012, 22:1, pp.~135-158.

\bibitem{MSbReview} M. V. Balashov. The gradient projection algorithm for a proximally smooth set
and a function with Lipschitz continuous gradient. Sbornik:
Mathematics, N 6, 2020. In press.

\end{thebibliography}
\end{document}